\documentclass[11pt,reqno]{amsart}
\usepackage{amsmath,amssymb,amsthm}
\usepackage{graphicx}
\usepackage[all]{xypic}
\usepackage[utf8,nocaptions]{vietnam}
\usepackage{calrsfs}
\input xypic
\def\thmsection{section}
\def\thmchangesection{changesection}

\def\thmchangechapter{changechapter}
\def\thmchange{change}
\def\thmplain{plain}
\ifx\theoremnumberstyle\thmplain
  \theoremstyle{break-italic}
  \newtheorem{satz}{Satz}
\else
  \ifx\theoremnumberstyle\thmsection
    \theoremstyle{break-italic}
    \newtheorem{satz}{Satz}[section]
  \else
    \ifx\theoremnumberstyle\thmchange
      \swapnumbers
      \theoremstyle{break-italic}
      \newtheorem{satz}{Satz}
    \else
      \ifx\theoremnumberstyle\thmchangesection
         \swapnumbers
         \theoremstyle{break-italic}
         \newtheorem{satz}{Satz}[section]
      \else
        \ifx\theoremnumberstyle\thmchangechapter
           \swapnumbers
           \theoremstyle{break-italic}
           \newtheorem{satz}{Satz}[chapter]
        \else
          \ifx\theoremnumberstyle\thmsection
             \theoremstyle{break-italic}
             \newtheorem{satz}{Satz}[section]
          \else
            \theoremstyle{break-italic}
            \newtheorem{satz}{Satz}[section]
          \fi
        \fi
      \fi
    \fi
  \fi
\fi

\theoremstyle{break-italic}
\newtheorem{theorem}[satz]{Theorem}

\newtheorem*{conjecture*}{Conjecture}

\theoremstyle{break-roman}

\newtheorem{remark}[satz]{Remark}

\newtheorem{conjecture}[satz]{Conjecture}

\theoremstyle{standard}

\theoremstyle{varthm-roman}
\newtheorem*{varthm-roman}{}
\theoremstyle{varthm-italic}
\newtheorem*{varthm-italic}{}
\theoremstyle{varthm-roman-break}
\newtheorem*{varthm-roman-break}{}
\theoremstyle{varthm-italic-break}
\newtheorem*{varthm-italic-break}{}
\theoremstyle{varthm-roman-no-punctuation}
\newtheorem{varthm-roman-no-punctuation-numbered}[satz]{}
\theoremstyle{varthm-italic-no-punctuation}
\newtheorem{varthm-italic-no-punctuation-numbered}[satz]{}

\newenvironment{varthm-roman-numbered}[1]{
  \begin{varthm-roman-no-punctuation-numbered}
    \mbox{\rm\textbf{#1}}
  }{\end{varthm-roman-no-punctuation-numbered}}
\newenvironment{varthm-italic-numbered}[1]{
  \begin{varthm-italic-no-punctuation-numbered}
    \mbox{\rm\textbf{#1}}
  }{\end{varthm-italic-no-punctuation-numbered}}
\newenvironment{varthm-roman-break-numbered}[1]{
  \begin{varthm-roman-no-punctuation-numbered}
    \mbox{\rm\textbf{#1}\newline}
  }{\end{varthm-roman-no-punctuation-numbered}}
\newenvironment{varthm-italic-break-numbered}[1]{
  \begin{varthm-italic-no-punctuation-numbered}
    \mbox{\rm\textbf{#1}}\newline
  }{\end{varthm-italic-no-punctuation-numbered}}
\numberwithin{equation}{section}
\newtheorem{proposition}{Proposition}[section]

\def\eex{{\accent"5E e}\kern-.385em\raise.2ex\hbox{\char'23}\kern-.08em}
\def\EES{{\accent"5E E}\kern-.5em\raise.8ex\hbox{\char'23 }}
\def\ow{o\kern-.42em\raise.82ex\hbox{
\vrule width .12em height .0ex depth .075ex \kern-0.16em \char'56}\kern-.07em}
\def\OW{O\kern-.460em\raise1.36ex\hbox{
\vrule width .13em height .0ex depth .075ex \kern-0.16em \char'56}\kern-.07em}

\def\Tr{{\rm Tr\,}}

\begin{document}

\title[Two trace inequalities for operator functions]{Two trace inequalities for operator functions}

\author{Trung Hoa Dinh}
\address{Trung Hoa Dinh, Department of Mathematics, Troy University, Troy, AL 36082, United States}
\email{thdinh@troy.edu}

\author{Toan Minh Ho }
\address{Minh Toan Ho, Institute of Mathematics, VAST, 18 Hoang Quoc Viet, Hanoi, Vietnam}
\email{hmtoan@math.ac.vn}

\author{Cong Trinh Le }
\address{Cong Trinh Le, Department of Mathematics, Quy Nhon University, 170 An Duong Vuong, Quy Nhon, Binh Dinh, Vietnam}
\email{lecongtrinh@qnu.edu.vn}

\author{Bich Khue Vo }
\address{Bich Khue Vo, University of Finance-Marketing, 2/4 Tran Xuan Soan, District 7, Ho Chi Minh City, Vietnam}
\email{votbkhue@gmail.com}


\date{\today}



\begin{abstract} In this paper we show that  for a non-negative operator monotone function $f$  on $[0, \infty)$ such that $f(0)= 0$ and for any positive semidefinite matrices $A$ and $B$,
\begin{equation*}\label{main2}
\Tr((A-B)(f(A)-f(B))) \le \Tr(|A-B|f(|A-B|)).
\end{equation*}
When the function $f$ is operator convex on $[0, \infty)$, the inequality is reversed.
\end{abstract}

\maketitle

\section{Introduction}

        For arbitrary nonnegative numbers $a \ge b$ and $p \ge 2$,
$$ 
a^{p-1} - b^{p-1} \ge (a-b)^{p-1}.
$$
Multiplying both sides of this inequality by $(a-b)$ we get
\begin{equation}\label{scalar}
    (a-b)(a^{p-1}-b^{p-1}) \ge (a-b)^p.
\end{equation}
For $p \in [1, 2]$, the inequality (\ref{scalar}) is reversed. Namely, we have
\begin{equation}\label{scalar2}
    (a-b)(a^{p-1}-b^{p-1}) \le (a-b)^p.
\end{equation}
From (\ref{scalar}) it implies that for $f, g \in  L_p(\Omega, \mu)$ (where $(\Omega, \mu)$ is some measure space),
\begin{equation}\label{func}
    \int_\Omega (f(x) - g(x)) (f(x)^{p-1} - g(x)^{p-1}) d\mu \ge  \int_\Omega (f(x) - g(x))^p d\mu.\end{equation}
In \cite{mus}, Mustapha Mokhtar-Kharroubi pointed out that this inequality may be used to get contractivity on the positive cone of $L_p(\Omega, \mu)$. Recently, Ricard   \cite{ricard} proved a non-commutative version of (\ref{func}) for von Neumann algebras. His result if translated into the language of matrices states that for $p \ge 2$ and for any $A, B \ge0$, 
\begin{equation}\label{ric}
\Tr((A-B)(A^{p-1} - B^{p-1})) \ge \Tr(|A-B|^p).
\end{equation}
Although the last inequality holds true, it is not obvious that the left hand side part is non-negative for any positive semidefinite matrices $A$ and $B$. That fact can be proved by using the Klein inequality \cite{carlen} which states that for a differentiable convex function $f$ on $(0, \infty)$,
$$ 
\Tr(f(A) - f(B)) \ge \Tr((A-B)f'(B)).
$$
Applying the Klein inequality for $t^p$ with $p \ge 1$ we obtain
$$ 
\Tr(A^p - B^p) \ge p\Tr((A-B)B^{p-1})
$$
and
$$
\Tr(B^p - A^p) \ge p\Tr((B-A)A^{p-1}).
$$
From the last two inequalities, we get
$$ 
\Tr(((A-B)(A^{p-1} - B^{p-1})) \ge 0.
$$

Recall a famous inequality for unitarily invariant norm  due to Ando \cite{ando}: For $p \ge 1$ and for any unitarily invariant norm $||| \cdot |||$, 
$$
||| A^p - B^p ||| \ge ||||A-B|^p|||.$$
Applying the above inequality for the trace norm $||A||_1 = \Tr(|A|),$ we obtain  
\begin{equation}\label{andoine}
\Tr(|A^p - B^p|) \ge \Tr(|A-B|^p).
\end{equation}

The inequality (\ref{ric}) attracts our attention because of the following reason: it provides an interpolation of the mentioned above Ando inequality.
\begin{proposition}\label{prop} Let $p \ge 2$. Then for any positive semidefinite matrices $A$ and $B$,
\begin{equation}\label{interpolation}
\Tr(|A^p - B^p|) \ge \Tr((A-B)(A^{p-1} - B^{p-1})) \ge \Tr(|A-B|^p).
\end{equation}
\end{proposition}
\begin{proof}
We prove the first inequality in (\ref{interpolation}) for any $p \ge 1$. In order to do that, let us recall the famous Powers-St\o rmer inequality in quantum hypothesis testing theory \cite{aud}: For any $A, B \ge 0$ and for any $s\in [0, 1]$,
\begin{equation}\label{powers}
\Tr(A+B - |A-B|) \le 2\Tr(A^sB^{1-s}).
\end{equation}
Applying (\ref{powers}) for $A^p$ and $B^p$ and for $s = 1/p$, we have
\begin{equation}\label{power1}
\Tr(A^p +B^p -|A^p-B^p|) \le 2\Tr(AB^{p-1}).
\end{equation}
Since $A$ and $B$ play the same role in the Powers-St\o rmer inequality, we also have
\begin{equation}\label{power2}
\Tr(A^p +B^p -|A^p-B^p|) \le 2\Tr(BA^{p-1}).
\end{equation}
From (\ref{power1}) and (\ref{power2}) we have
$$
\Tr(A^p +B^p -|A^p-B^p|) \le \Tr(AB^{p-1}) + \Tr(BA^{p-1}),
$$
or, 
$$
\Tr(|A^p - B^p|) \ge \Tr((A-B)(A^{p-1} - B^{p-1})).
$$
\end{proof}
\begin{remark} \rm  During the preparation of this paper, we received a comment from Dr. Ricard via a private communication. He provided a nice proof for the inequality (\ref{interpolation}) as follows. Recall that the Schatten $p$-norm is defined as $||A||_p = (\Tr(|A|^p))^{1/p}$. For $p\ge 1$, using Ando's inequality with  $\theta = 1/p$ and $\theta = (p-1)/p$, we get
$$
||A-B||_p \le ||A^p - B^p||_1^{1/p}, \quad ||A^{p-1} - B^{p-1}||_{p/(p-1)} \le ||A^p - B^p||_1^{p/(p-1)},
$$
where $A$ and $B$ are assumed to be positive semidefinite.
Consequently, 
\begin{align*}
\Tr((A-B)(A^{p-1} - B^{p-1})) & \le ||A-B||_p\cdot ||A^{p-1}-B^{p-1} ||_{p/(p-1)} \\
& \le ||A^p-B^p||_1. 
\end{align*}
\end{remark}

Now we should mention that the inequality (\ref{scalar}) is reversed when $p\in [1,2].$ Therefore, it is natural to ask whether the corresponding inequality holds for matrices.

At the same time, Ricard also gave us a short proof of the following inequality: For $p \in [1, 2]$ and for any positive semidefinite matrices $A$ and $B$, 
\begin{equation}\label{main1}
\Tr(|A-B|^p) \ge \Tr((A-B) (A^{p-1} - B^{p-1}))
\end{equation}
Indeed, 
\begin{align*}
\Tr((A-B)(A^{p-1} - B^{p-1})) &\le ||A-B||_p\cdot ||A^{p-1}-B^{p-1}||_{p/(p-1)} \\
&\le ||A-B||_p\cdot ||A-B||_p^{p-1} \\
& = \Tr(|A-B|^p),
\end{align*}
where we used the H\"older inequality in the first inequality, and the Ando inequality for $0 < \theta = p-1 <1$ and  $q \ge \theta$ as $|| A^\theta - B^\theta||_{q/\theta} \le ||A-B||_q^\theta$.

In this paper  we establish a generalization of (\ref{main1}) for operator monotone functions. Also the inequality (\ref{interpolation}) holds for operator convex functions instead of power functions $t^p$. 

\section{Main inequalities}
We should mention that for $p \in [1,2]$ the function $t^{p-1}$ is operator monotone  on $[0, \infty)$. Therefore, from  the inequality (\ref{main1}) it is interesting to know whether  the following inequality is true
$$
\Tr((A-B)(f(A)-f(B))) \le \Tr(|A-B|f(|A-B|)) 
$$
for some operator monotone  function $f$ under some conditions.

Based on the integral representation of operator monotone functions and operator convex functions we can establish a direct generalization of (\ref{main1}) for operator monotone functions on $[0, \infty)$.

\begin{theorem}
Let $f$ be a non-negative operator monotone function on $[0, \infty)$ such that $f(0)= 0$. Then for any positive semidefinite matrices $A$ and $B$,
\begin{equation}\label{main2}
\Tr((A-B)(f(A)-f(B))) \le \Tr(|A-B|f(|A-B|)).
\end{equation}
\end{theorem}
\begin{proof}
It is well-known (\cite{zhan})  that for any operator monotone function $f$ on $[0, \infty)$ there exists a positive measure $\mu$ on $[0, \infty)$ such that
$$
f(t) = \alpha +\beta t + \int_0^\infty \frac{st}{s+t}d \mu(s),
$$
where $\alpha =f(0)$ and $\beta \ge 0$. By the assumption of the theorem, $\alpha =0.$ Now, suppose that $A \ge B$ and put $C=A-B$. First we mention that 
\begin{equation}\label{identity}
    (B+s)^{-1} -(B+C+s)^{-1} = (B+s)^{-1}C(B+C+s)^{-1}.
\end{equation}
Therefore, we have
\begin{align*}
\Tr(A-B)(f(A)-f(B)) & =\Tr(\beta C^2) + \int_0^\infty \Tr(s^2C((B+s)^{-1} -(B+C+s)^{-1})) d\mu(s) \\ 
& = \Tr(\beta C^2) +\int_0^\infty \Tr(s^2C((B+s)^{-1}C(B+C+s)^{-1})) d\mu(s) \\
& \le \Tr(\beta C^2) + \int_0^\infty \Tr(sC^2(C+s)^{-1}) d\mu(s) \\ 
& = \Tr(Cf(C)),
\end{align*}
where the inequality follows from the fact that for any $s> 0$, $(B+s)^{-1} \le s^{-1}$ and $(B+C+s)^{-1} \le (C+s)^{-1}$, and  the positivity of $\Tr(XY) \ge 0$ for positive semidefinite matrices $X$ and $Y$.

In general, denote by $C_-$ and $C_+$ the negative and positive parts of $C$, respectively. Then we have $|A-B| = C_- +C_+,$ and $A-B=C_+ - C_-$. Put $Z=A+C_- = B + C_+.$ Then we have
\begin{align*}\label{8}
    \Tr((A-B)(f(A)-f(B))) & =  \Tr((A-Z)(f(A)-f(Z)))  + \Tr((A-Z)(f(Z)-f(B))) + \\ \nonumber
     & \quad + \Tr((Z-B)(f(Z)-f(B))) +\Tr((Z-B)(f(A)-f(Z))).   \end{align*}
     Using the fact that the function $f$ is operator monotone and $A, B \le Z$, one can see the second and the fourth terms in the last identity are negative. According to the previous case, we have
     \begin{align*}
              \Tr((A-Z)(f(A)-f(Z))) + \Tr((Z-B)(f(A)-f(Z))) & \le \Tr(C_-f(C_-))+\Tr(C_+f(C_+)) \\
              &= \Tr(|C|f(|C|)).
          \end{align*}
\end{proof}

\begin{remark} \rm  Combining inequality (\ref{main1}) with Ando's inequality, we have
$$
\Tr(|A^p-B^p|) \ge \Tr(|A-B|^p) \ge \Tr((A-B) (A^{p-1} - B^{p-1})).
$$
If we compare the last inequality with the inequality  (\ref{powers}) it turns out that the last one is an interpolation of the Powers-St\o rmer inequality for the power $s$ in $[1/2, 1]$. 
\end{remark}

Now let us give a generalization of Ricard's result for operator convex functions. 

\begin{theorem}
Let $f$ be a non-negative operator convex function on $[0, \infty)$ such that $f(0)= 0$. Then for any positive semidefinite matrices $A$ and $B$,
\begin{equation}\label{main3}
\Tr((A-B)(f(A)-f(B))) \ge \Tr(|A-B|f(|A-B|)).
\end{equation}
\end{theorem}
\begin{proof}
It is well-known (\cite{zhan}) that  for any operator convex function $f$ on $[0, \infty)$ there exists a positive measure $\mu$ on $[0, \infty)$ such that
$$
f(t) = \alpha +\beta t + \gamma t^2 +\int_0^\infty \frac{st^2}{s+t}d \mu(s),
$$
where $\alpha$ and $\beta$ are real and and $\gamma \ge 0$. By the assumption of the theorem, $\alpha = 0.$ Now, suppose that $A \ge B$ and put $C=A-B$. Therefore, 
\begin{align*}
\Tr((A-B)(f(A)-f(B))) & =\Tr(\beta C^2 + \gamma C((B+C)^2-B^2)) + \\
& \quad + \int_0^\infty s\Tr(C(C + s^2(B+C+s)^{-1} -s^2(B+s)^{-1})) d\mu(s) \\ 
&=\Tr(\beta C^2 + \gamma C(C^2+BC+CB)) + \\  
& \quad + \int_0^\infty s\Tr(C^2 - s^2C(B+C+s)^{-1}C(B+s)^{-1})) d\mu(s) \\ 
& \ge \Tr(\beta C^2 + \gamma C^3) + \int_0^\infty s\Tr(C^2 - sC^2(C+s)^{-1})) d\mu(s) \\ 
& = \Tr\left(C\left(\beta C + \gamma C^2 + \int_0^\infty s(C -s+s^2(C+s)^{-1}) d\mu(s)\right)\right) \\ 
& = \Tr(Cf(C)),
\end{align*}
where the inequality follows from the fact that for any $s> 0$, $(B+s)^{-1} \le s^{-1}$ and $(B+C+s)^{-1} \le (C+s)^{-1}$, and the positivity of $\Tr(XY) \ge 0$ for positive semidefinite matrices $X$ and $Y$.

In general, denote by $C_-$ and $C_+$ the negative and positive parts of $C$, respectively. Then we have $|A-B| = C_- +C_+,$ and $A-B=C_+ - C_-$. Put $Z=A+C_- = B + C_+.$ Then we have
\begin{align*}\label{8}
    \Tr((A-B)(f(A)-f(B))) & =  \Tr((A-Z)(f(A)-f(Z)))  + \Tr((A-Z)(f(Z)-f(B))) + \\ \nonumber
     & \quad + \Tr((Z-B)(f(Z)-f(B))) +\Tr((Z-B)(f(A)-f(Z))).   \end{align*}
 According to the previous case, we have
     \begin{align*}
              \Tr((A-Z)(f(A)-f(Z))) + \Tr((Z-B)(f(A)-f(Z))) & \ge \Tr(C_-f(C_-))+\Tr(C_+f(C_+)) \\
              &= \Tr(|C|f(|C|)).
          \end{align*}
To finish the proof, we need to show that the second and the fourth terms are positive. We again use the integral representation of operator convex functions and the fact that $C_-C_+ = 0$. We have
\begin{align*}
\Tr((A-Z)(f(Z)-f(A))) & = -\Tr(C_-(f(B+C_+)-f(B))) \\
& = -\Tr(\beta C_-C_+ + \gamma C_-((B+C_+)^2-B^2)) + \\
& \quad - \int_0^\infty s\Tr(C_-(C_+ + s^2(B+C_++s)^{-1} -s^2(B+s)^{-1})) d\mu(s) \\ 
& =  \int_0^\infty s^3\Tr(C_-((B+s)^{-1} - (B+C_++s)^{-1})) d\mu(s)  \\ 
& \ge 0.
\end{align*}
Similarly, we also have that the fourth term is positive. Thus, we finish the proof.
\end{proof}

To finish the paper, we would like to mention that the Ando inequality \cite{ando} was proved for general operator monotone functions and operator convex functions. Therefore, the following conjecture is natural.

\begin{conjecture} Let $||| \cdot|||$ be an arbitrary unitarily invariant norm and $f$ an operator monotone function on $[0, \infty)$ such that $f(0)=0$. Then for any positive matrices $A$ and $B$,
$$
||| (A-B)(f(A)-f(B))||| \le ||||A-B|f(|A-B|)|||.
$$

Also, the above inequality is reversed for an operator convex function $f$ on $[0, \infty)$ such that $f(0)=0$.
\end{conjecture}

\bigskip

\section*{Acknowledgments} We would like to express our sincere thanks to Dr. \'{E}ric Ricard for his comments and  valuable discussions that helped us to improve the preliminary version of the paper.

The second author is partially supported by Vietnam National Foundation for Science and Technology Development (NAFOSTED), grant no. 101.04-2017.12. \par The fourth  author is partially supported by Vietnam National Foundation for Science and Technology Development (NAFOSTED), grant no. 101.02-2017.310.


\end{document}